\documentclass[a4paper,11.5pt]{amsart}
\usepackage[a4paper, total={6in, 8in}]{geometry}
\usepackage[utf8]{inputenc}
\usepackage[usenames, dvipsnames]{color}
\newtheorem{theorem}{Theorem}
\newtheorem{definition}[theorem]{Definition}
\newtheorem{lemma}[theorem]{Lemma}

\newcommand{\C}{\mathbb C}
\newcommand{\R}{\mathbb R}
\newcommand{\N}{\mathbb N}

\newcommand{\D}{\mathbb D}
\newcommand{\T}{\mathbb T}
\newcommand{\B}{\mathbb B}

\usepackage{ulem}

\title[Boundary value problem for several complex variables]{Asymptotic first boundary value problem for holomorphic functions  of several complex variables}

\dedicatory{}

\author[P. M. Gauthier]{Paul M. Gauthier}
\address[Paul M. Gauthier]{D\'epartement de math\'ematiques et de statistique, Universit\'e de Montr\'eal, Montr\'eal, Qu\'ebec, Canada H3C3J7.}
\email{gauthier@dms.umontreal.ca}

\author[M. Shirazi]{Mohammad Shirazi}
\address[Mohammad Shirazi]{Department of Mathematics and Statistics, McGill University, Montr\'eal, Qu\'ebec, Canada, H3A 0B9}
\email{mohammad.shirazi@mail.mcgill.ca}

\begin{document}

\begin{abstract}
In 1955, Lehto showed that, for every measurable function $\psi$ on the unit circle $\mathbb T,$ there is function $f$ holomorphic in the unit disc $\D,$ having $\psi$ as radial limit  a.e. on $\T.$ We consider an analogous boundary value problem, where the unit disc is replaced by a Stein domain on a complex manifold and radial approach to a boundary point $p$ is replaced by (asymptotically) total approach to $p.$ 
\end{abstract}

\thanks{The present research was supported by NSERC (Canada) grant RGPIN-2016-04107.}

\maketitle

\section{Introduction}
The  boundary behaviour of holomorphic functions is an important topic in complex analysis (see for example \cite{No}). Due to a close connection to harmonic analysis, the most investigated type of boundary behaviour has been  the behaviour of holomorphic functions along paths tending to the boundary, especially along radii in the unit disc and radial limits have also been applied to the design of filters in electrical engineering (see e.g. \cite{AV}).

In 1955, O. Lehto \cite{L} showed that,  if 
$\psi_1(\theta),\psi_2(\theta):(0, 2\pi]\to [-\infty,+\infty]$ are arbitrary measurable functions,
then there exists a function $f(z),$ holomorphic in in the unit disc in the complex plane $\C$, such that
$$
	\lim_{\rho \nearrow 1} f(\rho e^{i\theta}) = \psi_1(\theta)+i\psi_2(\theta), \quad \mbox{for} \quad a.e. \quad  \theta\in[0,2\pi). 
$$
Lehto's theorem shows that the radial boundary values of  holomorphic functions in $\D$ can be prescribed almost everywhere on the boundary of the disc. On the other hand, any attempt to prescribe angular boundary values fails dramatically due to the  Luzin-Privalov uniqueness theorem \cite{No}, which asserts that if a function $f$ meromorphic in the unit disc $\D$ has angular limit $0$ at each point of a subset of the boundary having positive linear measure, then $f= 0.$ In Lehto's theorem, the topology on $[-\infty, +\infty],$ and consequently the $\sigma$-algebra of Borel sets, are induced by the natural mapping $[-1,+1]\to [-\infty, +\infty], \, t\mapsto \tan (t\pi/2).$

We shall denote by $A^\circ$ the interior of a set $A,$ with respect to a topology which will usually be evident from the context. For $p\in\C^n$ and $r>0,$ denote by $\B(p,r)$ the open ball of center $p$ and radius $r.$ For simplicity, we write $\B=\B(0,1).$ Let
$m$ denote the Lebesgue $2n$-measure on $\C^n=\R^{2n}$ and $\nu$ the standard measure of mass $1$ on $\partial \B.$

Whether Lehto's theorem on radial limits can be extended to the open ball in $\C^n$ is an open problem. However,  by replacing radial limits by limits along large sets (which we shall describe precisely), we shall state a theorem on boundary values which is valid on very general domains.  A special case is the following result for the unit ball.

\begin{theorem}\label{ball}
Let $\psi$ be a 
Borel measurable function on the boundary $\partial \B$  of the unit ball $\B\subset\C^n,$ 
whose restriction to some closed subset $F_0 \subset\partial \B$ is continuous.  
Then,   
there exists a holomorphic function $f$ on $\B,$ a set $F\subset \partial \B,$ with $F_0 \subset F$ and $\nu(\partial \B\setminus F)=0$ such that, for every $p\in F$,   there is a set $\mathcal F_p\subset \B,$ such that  
$$
	f(x)\to \psi(p), \quad  \mbox{as} \quad  x\to p, \quad x\in \mathcal F_p
$$
and 
$$
	\lim_{r\searrow  0}
	\frac{m\big(\B(p, r)\cap\,\mathcal F_p\big)}{m\big(\B(p,r)\cap\B\big)}=1.
$$
\end{theorem}

Although Theorem \ref{ball} resembles the theorem of Lehto, it is not strictly speaking a generalization of Lehto's result, for we are not claiming that the set of approach $\mathcal F_p$ contains a Stoltz cone at $p,$ nor even a radial segment at $p.$ {\it Au contraire}, all of the components of the $\mathcal F_p$ which we shall obtain will be compact, and of course a radial segment $\{rp: r_0\le r<1\}$ at $p$ cannot be contained in a compact subset of $\B.$  Thus, we may not prescribe angular approach, but we may at almost every point of the boundary prescribe approach from within  a set  whose complement  is asymptotically negligible with respect to Lebesgue measure.

Dmitry Khavinson has pointed out the following interesting instance of Theorem \ref{ball}. For $n=1,$ we take $F_0$ to be the union of two disjoint closed Jordan arcs $\alpha$ and $\beta$  and $\psi$ to be the characteristic function of $\alpha.$ We obtain a holomorphic function $f$ in the unit disc which has the striking property of approximating 1 at every point of $\alpha$  and  simultaneously approximating  0 at every point of $\beta,$  in the strong sense prescribed by the theorem.

Theorem \ref{ball} is a particular instance of the following theorem on complex manifolds, by which we mean connected holomorphic manifolds. 
Every domain $U$ in a complex manifold is automatically a complex manifold and we shall say that $U$ is is a Stein domain if $U$ is Stein qua manifold.
We shall be dealing with Borel measures and we take as part of the definition of a Borel measure that compact sets have finite measure.

\begin{theorem}\label{main}
Let $M$ be a complex manifold endowed with a distance $d$ and let $U\subset M$ be an arbitrary Stein domain. Let $\mu$ be a regular Borel measure on  $M$ and $\nu$ a regular Borel measure on $\partial U.$ Let $\psi$ be a 
Borel measurable function on $\partial U,$  
whose restriction to some closed subset $S\subset\partial U$ is continuous.  
Then,   
there exists a holomorphic function $f$ on $U,$ a set $F\subset \partial U,$ with $S\subset F$ and $\nu(\partial U\setminus F)=0$   such that, for every $p\in F$,   there is a set $\mathcal F_p\subset U,$ such that
$$
	f(x)\to \psi(p), \quad  \mbox{as} \quad  x\to p, \quad x\in \mathcal F_p
$$
and for every $\epsilon>0,$ there is an $r_{p,\epsilon}>0,$ such that
\begin{equation}\label{density}
	\mu\big( \B(p,r)\cap (U\setminus \mathcal F_p)\big)\le \epsilon\cdot\mu\big( \B(p,r)\cap U\big), 
		\quad \mbox{for all} \quad 0<r<r_{p,\epsilon}.
\end{equation}
\end{theorem}

\bigskip
Let us say that a set $A\subset U$ is of $\mu$-density zero relative to $U$ at a point $p\in\partial U$ if, for each $\epsilon>0,$ there is an $r_{p,\epsilon}>0,$ such that 
$$
	\mu\big( \B(p,r)\cap A\big)\le \epsilon\cdot\mu\big( \B(p,r)\cap U\big), 
		\quad \mbox{for all} \quad 0<r<r_{p,\epsilon}.
$$
By abuse of notation, we shall simply say that $A$ is of density zero at $p.$ Formula (\ref{density}) then is merely the assertion that the complement (in $U$) of  $\mathcal F_p$ is of density zero at $p.$ 
Let $\partial_0U$ be the set of those points $p\in \partial U$ such that $\mu(\B(p, r)\cap U)\neq 0$ for all $r>0$. For $p\in \partial_o U,$ we notice that $A$ is of density zero at $p,$ if and only if 
$$
\lim_{r\searrow 0}\frac{\mu\left(\B(p, r)\cap A\right)}{\mu\left(\B(p, r)\cap U\right)}=0.
$$

To prove Theorem \ref{main}, we henceforth fix such $M, d, U, \mu, \psi, S, \nu$ and we also fix some proper holomorphic embedding $h:M\to\C^n$ of $M$ into some complex Euclidean space $\C^n.$ 
We denote by $\mu_*$ the push-forward by $h$  of $\mu$ on $h(U).$ Since $h:U\to h(U)$ is a proper embedding, $\mu_*$ is a regular Borel measure on $h(U).$ By abuse of notation, we also denote by $\mu_*$ the trivial extension of $\mu_*$ to Borel subsets of $\C^n.$ That is, for every Borel subset $A\subset\C^n,$ we put $\mu_*(A)=\mu_*\big(A\cap h(U)\big).$  Since $h(U)$ is closed in $\C^n,$ it follows that the extension $\mu_*$ is also a regular Borel measure on $\C^n.$

The boundary value problem, which we discuss in Theorem \ref{main} for holomorphic functions of several complex variables, was considered in \cite{FG1} for holomorphic functions of a single complex variable. The analogous problem for harmonic functions on Riemannian manifolds was treated in \cite{FG2}.  The methods employed in these two papers were different and in \cite{FG3} a unified approach for holomorphic functions of a single variable and harmonic functions of several variables is presented by employing the Lax-Malgrange Theorem expressed in the language of vector bundles.
The Lax-Malgrange Theorem is for differential operators between bundles of equal rank. The Cauchy-Riemann operator on a complex manifold of dimension $n$ is a mapping from a bundle of rank $1$ to a bundle of rank $n,$ so we only have equal rank when $n=1.$  
For several complex variables, we thus require a different technique and to this end we shall introduce``bricks" in Section \ref{bricks}.


\section{Brick Constructions}\label{bricks}

A compact set $K\subset\C^n$ is said to be {\it polynomially convex} if, for every $z\not\in K,$ there is a holomorphic polynomial $p,$ such that $|p(z)|>|p(\zeta)|,$ for every $\zeta\in K.$ By the Oka-Weil theorem, if $K$ is polynomially convex, then every holomorphic function on $K$ can be uniformly approximated by holomorphic polynomials.

The most basic domains in $\C^n$ are balls and polydiscs. Closed balls and closed discs are polynomially convex and so are good sets on which to approximate holomorphic functions by holomorphic polynomials. 
 However, finite unions of disjoint closed balls or polydiscs are not in general polynomially convex. The fundamental tool regarding the polynomial convexity of the union of two sets is the Separation Lemma of Eva Kallin which asserts that the union of two polynomially convex compacta is polynomially convex if they can be separated by a (real) hyperplane. 
It turns out that (real) bricks, that is, bricks in $\C^n$ viewed as $\R^{2n}$ are more convenient, in that it is possible to approximate on finite unions of disjoint closed bricks, provided they are appropriately ``stacked", taking into account the separation lemma. Approximation on appropriately stacked unions of closed balls and closed polydiscs is also possible, but stacks of bricks have the advantage that they can be stacked ``more tightly".
Approximation on bricks was exploited in \cite{BG} to study universality, but the term bricks was not employed there.

\subsection{Bricks, Stratification}

Consider a brick $B\subset\C^n:$
$$B=\left\{(x_1, y_1, \dots, x_n, y_n) : a_t\leq x_t\leq b_t,\, c_t\leq y_t\leq d_t,\, t=1, \dots, n\right\}.$$
We wish to partition $B$ into smaller (sub)bricks by inserting some real hyperplanes perpendicular to the coordinate axes $\left\{x_1,y_1,\ldots,x_n,y_n\right\}$. Suppose the real hyperplanes orthogonal to the coordinates are 
\begin{equation}\label{hyperplaneequation}
\begin{split}
a_1=\lambda_{1,0}&<\lambda_{1,1}<\lambda_{1,2}<\cdots<\lambda_{1, r_1}<\lambda_{1,r_1+1}=b_1\\
c_1=\mu_{1, 0}&<\mu_{1, 1}<\mu_{1, 2}<\cdots<\mu_{1, s_1}<\mu_{1, s_1+1}=d_1\\
&\cdots\hspace{4cm}\cdots\\
a_n=\lambda_{n,0}&<\lambda_{n,1}<\lambda_{n,2}<\cdots<\lambda_{n, r_n}<\lambda_{n, r_n+1}=b_n\\
c_n=\mu_{n, 0}&<\mu_{n, 1}<\mu_{n, 2}<\cdots<\mu_{n, s_n}<\mu_{n, s_n+1}=d_n
\end{split}
\end{equation}
for some $r_1, s_1, \cdots, r_n, s_n\in \N$ and $\lambda_{i, j}, \mu_{i, j}\in \R$. 
These hyperplanes partition $B$ into finitely many bricks $(\beta_{i, j}),$ where the index $i, j$ is clarified as follows
\begin{equation}\label{partitioncoordinates}
\begin{split}
\beta_{i,j}&= \beta_{i_1, \dots, i_n, j_1, \dots, j_n}\\
&=\left\{(x_1, y_1, \dots, x_n, y_n) : \lambda_{t, i_t}\leq x_t\leq \lambda_{t, (i_t+1)},\,\mu_{t, j_t}\leq y_t\leq \mu_{t, (j_t+1)},\, t=1, \dots, n\right\},
\end{split}
\end{equation}
where $0\leq i_t\leq r_t$, and $0\leq j_t\leq s_t$.

We also introduce bricks $\widetilde{\beta}_{i,j}\subset\beta_{i,j}^{\mathrm{o}},$  $\widetilde{\beta}_{i,j}=\widetilde{\beta}_{i_1, \dots, i_n, j_1, \dots, j_n}$, obtained by inserting two new parallel hyperplanes between each consecutive parallel pair from the previous hyperplanes. That is, the $x_t$ and $y_t$ coordinates of $\widetilde{\beta}_{i,j}$ are given by
\begin{equation}\label{hyperplanes}
\begin{split}
\lambda_{t, i_t}&<\lambda'_{t, i_t}\leq x_t\leq\lambda''_{t, (i_t+1)}<\lambda_{t, (i_t+1)},\\
\mu_{t, j_t}&<\mu'_{t, j_t}\leq y_t\leq \mu''_{t, (j_t+1)}< \mu_{t, (j_t+1)},
\end{split}
\end{equation}
In other words, we used one prime for a hyperplane inserted to the right hand side of a fixed hyperplane $x_t=\lambda_{ti_t}$ (or $y_t=\mu_{tj_t}$) and a double prime for a hyperplane added to the left hand side of $x_t=\lambda_{ti_t}$  (or $y_t=\mu_{tj_t}$). For such a construction, we say that the family of bricks $(\widetilde{\beta}_{i,j})$  is a \textit{shrinking} of the family of bricks $(\beta_{i,j})$ and we call the compactum $\widetilde  B=\cup\widetilde{\beta}_{i,j}$ a {\it stratification} of the brick $B$.

\begin{lemma}\label{blocksPC}
Let $B\subset\C^n$ be a brick,  let $(\beta_{i,j})$ be a partition of $B$ by bricks induced by  hyperplanes orthogonal to the real coordinates of $\C^n.$ Then, every stratification $\widetilde B=\bigcup_{i,j}\widetilde{\beta}_{i,j}$ is polynomially convex. 
\end{lemma}

\begin{proof}
We may use the superscript (n) to indicate that a set $A^{(n)}$ is in $\C^n$, e.g. $B^{(n)}, \, \beta^{(n)}_{i, j},$ etc.

If $n=1$, then the proof is trivial. A finite union of disjoint compacta having connected complements is again of
connected complement and hence polynomially convex. Being separated by hyperplanes is even redundant in this case. So we assume the claim is true for $n>1$. Let $B^{(n+1)}$ be a brick in $\C^{n+1},$  let $(\beta_{i,j}^{(n+1)})$ be a partition of $B^{(n+1)}$ by bricks as above and $(\widetilde{\beta}_{i,j}^{(n+1)})$ a shrinking thereof. We have
$$
	\beta_{i,j}^{(n+1)}=
\beta_{i_1,\ldots,i_{n+1},j_1,\ldots,j_{n+1}}^{(n+1)}=
\beta^{(n)}_{i_1,\ldots,i_n,j_1\ldots,j_n}\times\beta^{(1)}_{i_{n+1},j_{n+1}}=
	\beta^{(n)}_{i,j}\times\beta^{(1)}_{i_{n+1},j_{n+1}}
$$
and for the shrinking
$$
	\widetilde{\beta}^{(n+1)}_{i,j}=\widetilde{\beta}^{(n)}_{i,j}\times\widetilde{\beta}^{(1)}_{i_{n+1},j_{n+1}}.
$$
Thus, 
$$
	\widetilde B^{(n+1)}=\widetilde B^{(n)}\times \widetilde B^{(1)}, \quad \mbox{ where } \quad \widetilde B^{(n)}=\cup_{ i, j}\widetilde{\beta}^{(n)}_{i, j}, 
	\quad \widetilde B^{(1)}=\cup\widetilde{\beta}^{(1)}_{i_{n+1}, j_{n+1}}.
$$
$\widetilde B^{(1)}$ is polynomially convex by the trivial case $n=1$ and $\widetilde B^{(n)}$ is polynomially convex by the induction hypothesis. Thus, $\widetilde B^{(n+1)}$ is polynomially convex as a product of polynomially convex sets. 

\end{proof}

The following lemma shows that a stratification of a block can be chosen such that the interior of the block minus its stratification has arbitrarily small measure. Consequently, in the subsequent lemmas, we show that the density of the pre-image  (under $h$) can also be made  arbitrary small.

\begin{lemma}\label{stratification}
Given an arbitrary brick $B$,  an arbitrary $\delta>0,$
and an arbitrary regular Borel measure $\eta$ on $\C^n,$
there exists a partition $(\beta_{i,j})$ of $B$ such that, for every brick  $\beta_{i,j}$ of the partition, $diam\,\beta_{i,j}<\delta$ and there exists a stratification $\widetilde B=\cup\widetilde{\beta}_{i,j}$ of this partition, such that 
$$
	\eta\big((B\setminus \widetilde B)^{\mathrm{o}}\big) = \eta\left(\bigcup_{i,j}\,(\beta_{i,j}^{\mathrm{o}}\setminus\widetilde\beta_{i,j})\right)<\delta.
$$
\end{lemma}

\begin{proof}
We may denote
$$B=\left\{(x_1, y_1, \dots, x_n, y_n) : a_t\leq x_t\leq b_t,\, c_t\leq y_t\leq d_t,\, t=1, \dots, n\right\}.$$
The distance between every two points in  $(\beta_{i,j})$, given by  (\ref{partitioncoordinates}), is dominated by 
\begin{equation}\label{diameter}
\sum_{t=1}^n
\left((\lambda_{t, (i_t+1)}-\lambda_{t, i_t}) + (\mu_{t, (j_t+1)}-\mu_{t, j_t})\right).
\end{equation}

Since there are at most $4n$ values $a_t, b_t, c_t, d_t$, $t=1, \ldots, n,$ we may choose the coordinates of the  $\beta_{i,j}$'s in (\ref{hyperplaneequation}) to ensure that (\ref{diameter}) is less than $\delta$, so we obtain that $diam\,\beta_{i,j}<\delta$, for each $\beta_{i,j}.$

For the stratification $\widetilde B$ of  $B$, with the  coordinates in (\ref{hyperplanes}), for the additional hyperplanes, which we insert,  we have
$$
(B\setminus \widetilde B)^{\mathrm{o}}\subset\bigcup_{t=1}^n \left(
(\bigcup_{i=0}^{r_t+1}X_{t, i}) \cup (\bigcup_{j=0}^{ s_t+1}Y_{t, j}
)\right),
$$
for the open bricks $X_{i,j}$ and $Y_{i,j}$ given by 
\begin{equation}\label{XY}
\begin{split}
X_{t, 0}&=B^{\mathrm{o}}\cap
\left(\lambda_{t, 0}< \, x_t <\lambda'_{t, 0}\right), \\
X_{t, i}&= B^{\mathrm{o}}\cap
\left(\lambda_{t, i}''< \, x_t <\lambda'_{t, i}\right), i=1, \dots, r_t \\
X_{t, (r_t+1)}&=B^{\mathrm{o}}\cap
\left(\lambda_{t, (r_t+1)}''< \, x_t <\lambda_{t, (r_t+1)}\right), \\
Y_{t, 0}&=B^{\mathrm{o}}\cap
\left(\mu_{t, 0}< \, y_t <\mu'_{t, 0}\right), \\
Y_{t, j}&= B^{\mathrm{o}}\cap
\left(\mu_{t, j}''< \, y_t <\mu'_{t, j}\right), j=1, \dots, s_t \\
Y_{t, (s_t+1)}&= B^{\mathrm{o}}\cap
\left(\mu_{t, (s_t+1)}''< \, y_t <\mu_{t, (s_t+1)}\right),\\
\end{split}
\end{equation}
for $t=1, \dots, n$. 
Since the measure $\eta$ is regular, 
$$
	 \eta(X_{t, i})\longrightarrow 0, \quad  \mbox{as}\quad \lambda_{t, i}^\prime\searrow \lambda_{t, i} 		\quad \mbox{and} {\quad \lambda''_{t, (i+1)}}\nearrow\lambda_{t, (i+1)};
$$
$$
	 \eta(Y_{t, j})\longrightarrow 0, \quad  \mbox{as}\quad \mu'_{t, j}\searrow\mu_{t, j}
	\quad \mbox{and} \quad	\mu''_{t, (j+1)}\nearrow\mu_{t, (j+1)} .
$$
We may thus choose the  hyperplanes in (\ref{hyperplanes}), so that 
\begin{equation*}
\begin{split}	
	\eta\left((B\setminus \widetilde B)^{\mathrm{o}}\right) &\le \sum_{t=1}^n \eta\left(
(\bigcup_{i=0}^{r_t+1}X_{t, i}) \cup (\bigcup_{j=0}^{s_t+1}Y_{t, j}
)\right)\\
&\le \sum_{t=1}^n \left(\sum_{i=0}^{r_t+1}\eta(X_{t, i})+\sum_{j=0}^{s_t+1}\eta(Y_{t, j})\right)< \delta. 
\end{split}
\end{equation*}
\end{proof}

\begin{lemma}\label{predensity}
Given an arbitrary brick $B,$  and an arbitrary $\delta>0,$ there exists a partition $(\beta_{i,j})$ of $B$ such that, for every brick  $\beta_{i,j}$ of the partition, $diam\,\beta_{i,j}<\delta$ and
given a compact set $K\subset \partial_0 U$, for every $0<r<t,$
there exists a stratification $\widetilde B=\cup\widetilde{\beta}_{i, j}$ of this partition, such that
$$
	\mu(\B(p, s)\cap h^{-1}\big((B\setminus \widetilde B)^{\mathrm{o}}\big))\le \delta\cdot\mu\big(\B(p, s)\cap U\big), 
	\quad \forall p\in K, \quad \forall s\in [r,t]. 
$$
\end{lemma}

\begin{proof}
Suppose a brick $B$, $K,$ $\delta>0$ and  $0<r<t$ are given. Proof of the first part is the same as Lemma \ref{stratification}. To prove the second part, define
$$
I_t=\left\{q\in K : \B(q, t)\cap h^{-1}(B)\neq \emptyset \right\}.
$$
For $p\in K\setminus I_t,$ the result is trivial, since for such $p$
and for all $s\in [r, t]$ one has 
$$\B(p, s)\cap h^{-1}(B\setminus \widetilde B)\subseteq \B(p, t)\cap h^{-1}(B)=\emptyset.$$

Now, suppose $p\in I_t$. 
We claim that, if
$$M:=\inf_{q\in I_t} \mu \left(\B(q, r) \cap U \right)\neq 0,$$
then the proof is complete. Here is the reasoning. 
Invoking Lemma (\ref{stratification}), for the push forward $\mu_*$ of $\mu,$ there exists a stratification $\widetilde B$ of $B$ such that
\begin{equation}
\mu ( h^{-1}((B\setminus \widetilde B)^{\mathrm{o}}) )=\mu_*((B\setminus \widetilde B)^{\mathrm{o}})\leq \delta M.
\end{equation}
Therefore, for $p\in I_t$, and $s\in [r, t]$, the above implies that
\begin{equation*}
\begin{split}
\mu (\B(p,s)\cap h^{-1}((B\setminus \widetilde B)^{\mathrm{o}}) )\le \mu (h^{-1}\big((B\setminus \widetilde B)^{\mathrm{o}}\big) )\le \delta M\leq \delta \cdot \mu \left(\B(p, r) \cap U \right)\leq \delta \cdot \mu \left( \B(p, s) \cap U \right).
	\end{split}
\end{equation*}

Now we show that $M\neq 0$. Suppose $M=0$, then there exists a sequence $(p_j)$ of elements in $I_t\subseteq K$ such that 
$$\lim_{j\rightarrow \infty}\mu\left(\B(p_j, r) \cap U\right)=0.$$
By the compactness of $K$, there is a point $p\in K$ and a subsequence of $(p_j)$ converging to $p$. By abuse of notation, we may assume this subsequence is $(p_j)$. For an arbitrary $\epsilon>0$, one can choose $j_\epsilon\in \N$ large enough such that $d(p_j, p)<r/2$ and $\mu\left(\B(p_j, r) \cap U\right)<\epsilon$ for every $j>j_\epsilon.$ One, therefore, has
$$
\mu\left(\B(p, r/2) \cap U\right)\leq \mu \left(\B(p_j, r) \cap U\right)\leq \epsilon.
$$
As epsilon was arbitrary, we can conclude that $\mu\left(\B(p, r/2) \cap U\right)=0$; that is, we have shown that there exists a neighbourhood of  $p$
 such that the measure of its intersection with $U$ is zero, which contradicts our assumption that $p\in \partial_0U.$
\end{proof}


\subsection{Serpentine Tiling and Direction}
We can fix an order on the $4n$ directions given by the  positive  and negative $x_k$-directions and the positive and negative $y_k$-direction, $ k=1, \dots, n$; that is, the order 
$$+x_1, -x_1,+y_1, -y_1, \cdots, +x_n, -x_n, +y_n, -y_n.$$ 
Rename these as the $k$-directions, where $k=1,2,\ldots, 4n.$
For $k\in\N,$ we define the $k$-direction to mean the $k$ (mod $4n$) direction.

\begin{definition} A brick $B$ has $4n$ faces. Let us call the face whose outward normal is in the $k$-direction, the $k$-side of $B.$ We define the $k$-reflection $k(B)$ to be the reflection of $B$ in the $k$-side of $B.$  We extend this definition to $k\in \N$ (mod $4n$). We define the $k$-double $D_k(B)$ of the brick $B$ to be the brick $B\cup k(B).$
\end{definition}

If $B$ is a brick and $k=1,2,\ldots$, then the $k$-double $D_k(B)$ is polynomially convex as it is a brick.

Let us construct a special sequence of bricks $B_k, \, k=1,2, \dots$. We take $B_1$ to be an arbitrary brick in $\C^n$. Let $B_2$ be the $1$-reflection $1(B_1)$ of $B_1.$ We put $B_3=2(B_1\cup B_2).$  
Having defined $B_1,\dots ,B_\ell$, we define $B_{\ell+1}$ as  $k(B_1 \cup B_2 \cup . . .\cup B_\ell)$, where $k=\ell$ (mod $4n$).

\begin{lemma}\label{tiling} For an arbitrary brick $B_1$ in $\C^n,$ the family of bricks $(B_k)_{k=1}^\infty$ defined as above forms a tiling of $\C^n.$ That is, for $k\not=j, \, B_k^{\mathrm{o}}\cap B_j^{\mathrm{o}}=\emptyset$ and $\cup_k B_k=\C^n$.
\end{lemma}

\begin{proof}
It is clear that for $k>j, \, B_k^{\mathrm{o}}\cap B_j^{\mathrm{o}}=\emptyset,$ since every $B_k, k>1,$ is a reflection of the brick $\cup_{j<k} B_j.$

We may write
$$
	\cup_{j=1}^kB_j =\left \{(x_1, y_1, \dots, x_n, y_n):a_t^{(k)}\le x_t\le b_t^{(k)}, \, c_t^{(k)}\le y_t\le d_t^{(k)}, \, t=1,\ldots,n\right \}.
$$
To show that we have a tiling, it is sufficient to show that, for $t=1,\ldots,n,$
\begin{equation}\label{tiling}
a_t^{(k)}\searrow -\infty, \quad b_t^{(k)}\nearrow+\infty, \quad  c_t^{(k)}\searrow -\infty, \quad d_t^{(k)}\nearrow+\infty,
\quad \mbox{as} \quad k\to+\infty.
\end{equation}
Note that $\cup_kB_k$ forms a tiling of $\C^n$ if and only if $\cup_k (B_k+v)$ where $v$ is a vector in $\C^n$ is a tiling; that is, a translation of a tiling is a tiling. 
Thus, we may assume that $0\in B_1^{\mathrm{o}},$  equivalently,
$$
		a_t^{(1)}<0<b_t^{(1)}, \quad c_t^{(1)}<0<d_t^{(1)}, \quad \mbox{for} \quad t=1,\ldots,n.
$$

To show (\ref{tiling}), we begin by showing that 
\begin{equation}\label{ab1k}
	\lim_{k\to+\infty}a_1^{(k)}=-\infty \quad \mbox{and}	\quad  \lim_{k\to+\infty}b_1^{(k)}=+\infty. 	
\end{equation}

The first reflection is in the $(+x_1)$-direction, so $a_1^{(2)}=a_1^{(1)}$ and $b_1^{(2)}=b_1^{(1)}+(b_1^{(1)}-a_1^{(1)})>2b_1^{(1)},$ bearing in mind that $a_1^{(1)}<0.$ The second reflection is in the $(-x_1)$-direction, so $b_1^{(3)}=b_1^{(2)}$ and $a_1^{(3)}=a_1^{(2)}-(b_1^{(2)}-a_1^{(2)})<-2b_1^{(1)}.$
One should note that the next $4n-2$ reflections have no contribution to the $x_1$ axis; that is, $a_1^{(k)}=a_1^{(3)}$ and $b_1^{(k)}=b_1^{(3)},$ for $k=4,5,\ldots,4n.$ Thus, $a_1^{(4n)}<-2b_1^{(1)}$ and $2b_1^{(1)}<b_1^{(4n)}.$ Repeating the same argument for the next $4n$ reflections, we have that $a_1^{(2\cdot 4n)}<-2b_1^{4n}<-2^2b_1^{(1)}$ and $2^2b_1^{(1)}<2b_1^{(4n)}<b_1^{(2\cdot 4n)}.$ By recursion, we conclude that
$$
	a_1^{(j\cdot 4n)}<-2^jb_1^{(1)} \quad \mbox{and} \quad 2^jb_1^{(1)}<b_1^{(j\cdot 4n)},
	\quad \mbox{for} \quad j=1,2,\ldots. 
$$
Since the sequences $(a_1^{(k)})$ and $(b_1^{(k)})$ are monotone, this establishes (\ref{ab1k}).

The third and fourth reflections are respectively in the $(+y_1)$ and $(-y_1)$ directions. An argument analogous to the one just given shows that 
$$
	c_1^{(j\cdot 4n)}<-2^j c_1^{(1)} \quad \mbox{and} \quad 2^jd_1^{(1)}<d_1^{(j\cdot 4n)},
	\quad \mbox{for} \quad j=1,2,\ldots, 
$$
which implies that $c_1^{(k)}\searrow-\infty$ and $d_1^{(k)}\nearrow+\infty$  as  $k\to+\infty$. 

By the symmetry of the  construction and the fact that the directions change periodically each $4n$ times, the same argument can be applied for $t=2,\ldots,n,$  which yields (\ref{tiling})

\end{proof}

\begin{definition}[Serpentine and Masonic tilings]
The tiling $(B_k)_{k=1}^\infty$ of $\C^{n}$ by the above process is called a serpentine tiling. If to each $B_k,\, k=1,2,\ldots,$ we associate a stratification $\widetilde{B}_k,$ we call the reduced tiling $(\widetilde{B}_k)_{k=1}^\infty$ a Masonic tiling. We shall call  the closed set $E=\cup_k \widetilde B_k$ the Masonic temple constructed from the Masonic tiling ($\widetilde B_k)_{k=1}^\infty.$ 
\end{definition}

For the remainder of this paper $(B_k)_{k=1}^\infty$ will denote an arbitrary but fixed serpentine tiling of $\C^n.$ We denote by $(\beta_{i,j}^ k)_{k=1}^\infty$ an associated sequence of tilings. That is, $(\beta_{i,j}^ k)_{i, j}$ is a tiling of $B_k,$ for $k=1,2, \ldots .$ One should note that $(\widetilde B_k)_{k=1}^\infty$ does not cover $\C^{n}$. 

We now generalize, in some sense, Lemma \ref{predensity} from a block to a tiling of $\C^n$.

\begin{lemma}\label{9-12}
Let $(B_k)_{k=1}^\infty$ be a serpentine tiling of $\C^n;$  
let  $(\delta_k)_{k=1}^\infty$ be a sequence of positive numbers: 
let $\eta$ be an arbitrary regular Borel measure on $\C^n$ and let $(K_k)_{k=1}^\infty$ be an exhaustion of $\partial_0 U$ by compacta. Then, there exists a partition $(\beta_{i,j}^k)$ of $B_{k}$, and an associated Masonic tiling $(\widetilde B_k)_{k=1}^\infty,$ whose corresponding Masonic temple $E=\cup_k\widetilde B_k$ is such that, for $k=1,2,\ldots,$

(a) $diam(\beta^k_{i,j})<\delta_k$, for each brick $\beta^k_{i,j}$ of  $B_k;$

(b)	$\eta\left\{z\in\C^n\setminus E: |z|>k\right\}< \delta_k$ and hence:
$$
	\mu \left\{x\in U: h(x)\in \C^n \setminus E \textit{ and }  |h(x)|>k\right\}< \delta_k.
$$

(c) for all $p\in\partial U$ and for all $\lambda>0$, there exists an $r_{p, \lambda}>0$, such that 
$$
	\mu\big(\B(p, r)\cap(U\setminus h^{-1}(E))\big) \le 
		\lambda\cdot\mu\big(\B(p, r)\cap U\big), \quad \mbox{for all} \quad  0<r<r_{p, \lambda}.
$$
\end{lemma}

\begin{proof}
Let $(\epsilon_k)_{k=1}^\infty$  be an arbitrary sequence of positive numbers. To prove (a) and (b), 
let $B=B_k$, $\delta=\min \{\delta_k, 2^{-k}, \epsilon_k\}$, $K=K_k$, and $[r, t]=[1/(\ell+1),1/\ell] , k=1, 2, \ldots$.  
Since both Lemma \ref{stratification}, and Lemma \ref{predensity} have the property that every finer partition satifies the same conclusion, we may thus choose a partition  $(\beta_{i,j}^k)$ of $B_{k},$ such that, for $k=1,2,\ldots$,
\begin{equation}
	diam(\beta^k_{i,j})<\delta_k, \quad \mbox{for each brick} \quad \beta^k_{i,j} \quad \mbox{of} \quad B_k;
\end{equation}
\begin{equation}\label{epsilon}
\eta((B_{k}\setminus\widetilde B_{k})^{\mathrm{o}})<\epsilon_k,
\end{equation} 
which proves (a) and 
\begin{equation}\label{muineq}
	\mu\left(\B(p, s)\cap h^{-1}((B_k\setminus \widetilde B_k)^{\mathrm{o}})\right)\le 2^{-k}\cdot \mu\left(\B(p, s)\cap U\right); 
	\,\,\, p\in  K_k, \, \ell, \, s\in \left[\frac{1}{\ell+1}, \frac{1}{\ell}\right]. 
\end{equation}

Now, since the family $(\widetilde{B}_{k})_{k=1}^\infty$ of stratifications and the family $(\widetilde \beta_{i,j}^k)$ of all the bricks of all the $\widetilde B_{k}$ are both locally finite, we may choose a sequence $(\epsilon_k)_{k=1}^\infty$ which tends to zero sufficiently rapidly that (\ref{epsilon}) yields the desired estimates
$$
	\eta\left\{z\in\C^n\setminus E: |z|>k\right\}< \delta_k.
$$ 
The second part of (b) can be proved by noting that $\mu_*$ is a regular Borel measures on $\C^n$.

To prove (c), note that for $p\in\partial U\setminus \partial_0 U,$ this is trivial, since both sides are zero. Now for  $p\in \partial_0 U,$ there exists $k_p\in \N$ such that $p\in K_{k_p}$. Moreover, $p\in K_{k}$, for all $k \ge k_p$. 
Now suppose $\lambda>0$ and suppose $0<r<1.$
We may choose $k(\lambda)\ge k_p$ such that $2^{-(k(\lambda)-1)}<\lambda.$ Furthermore, choose $\ell_\lambda$ so large that $\B(p, r)\cap h^{-1}(B_k)=\emptyset,$ for all $k\le k(\lambda)$ and all $r<1/\ell_\lambda=:r_{p, \lambda}.$ 
Then, for all $r<r_{p, \lambda}$, by subadditivity of the measure $\mu$ and inequality (\ref{muineq}), one has
\begin{equation*}
\begin{split}
	\mu\big(\B(p, r)\cap(U\setminus h^{-1}(E))\big)&=
	\mu\left(\bigcup_{k=1}^\infty\big\{\B(p, r)\cap h^{-1}\big((B_k\setminus\widetilde B_k)^{\mathrm{o}}\big)\big\}\right)\\
	&\le \sum_{k= k(\lambda)}^\infty \mu \left( \B(p, r)\cap h^{-1}\big((B_k\setminus\widetilde B_k)^{\mathrm{o}}\big)\right)\\
	&\le \sum_{k=k(\lambda)}^\infty 2^{-k}\mu(\B(p, r)\cap U)\\
	&\le \lambda \cdot \mu(\B(p, r)\cap U),
\end{split}
\end{equation*}
which completes the proof in (c).
\end{proof}
\section{Approximation in $\C^n$}

The first approximation theorem is about approximating a sequence of holomorphic functions defined respectively  on  a sequence of tiles by an entire function. Then we shall approximate (outside of a small set) a continuous function on $\C^n$ by an entire function.

\begin{theorem}\label{masonic approximation} 
Let  $(\widetilde{B}_k)_{k=1}^\infty$ be a Masonic tiling in $\C^n,$ $(f_k)_{k=1}^\infty$ a sequence of functions, $f_k$ holomorphic on $\widetilde{B}_k$, and  $(\epsilon_k)_{k=1}^\infty$ a sequence of positive numbers. Then,   there is an entire function $g$, such that 
\begin{equation*}
|g(z)-f_k(z)|<\epsilon_k \quad \mbox{for all} \quad z\in \widetilde B_k, \quad  k=1,2,\ldots.
\end{equation*}
\end{theorem}

\begin{proof}
We may assume that, for each $k,$ $\sum_{j>k}\epsilon_j<\epsilon_k.$ 

Set $E_0=\emptyset$,  and $E_n=\cup_{k=1}^nB_k, n>0$. 
By definition, $E_1=B_1$ is polynomially convex. Since $E_n$, for $n>1$, is the $n$-double of $E_{n-1}$, it follows that $E_n$ is polynomially convex. We shall construct by induction a sequence $\{P_n\}$ of polynomials, such that $|P_k-f_k|<\epsilon_k/2^k$ for $k=1,2,\ldots,$ on $\widetilde B_k$ and $|P_k-P_{k-1}|<\epsilon_k/2^k$ on $E_{k-1}$.
The function $f_1$ is holomorphic on $\widetilde{B}_1$ and $\widetilde{B}_1$ is polynomially convex, so there is a holomorphic polynomial $P_1$ such that 
$$|P_1(z)-f_1(z)|<\frac{\epsilon_1}{2^1}, \quad \textit{ for all }\quad z\in\widetilde{B}_1.$$
Define 
\begin{equation*}
	g_{2}(z) = 
\left\{
	\begin{array}{ll}
		P_1(z),	&	z\in E_1=B_1,\\
		f_2(z),	&	z\in \widetilde{B}_2,	
	\end{array}
\right.
\end{equation*}
which is a holomorphic function on $B_1\cup \widetilde{B}_2$. By the Kallin Separation Lemma, $B_1\cup \widetilde{B}_2$ is polynomially convex, so there exists a holomorphic polynomial $P_2$ such that 
$$|P_2(z)-g_2(z)|<\frac{\epsilon_2}{2^2}, \quad \textit{ for all }\quad z\in B_1\cup \widetilde{B}_2.$$
Having defined $P_1, \dots, P_k$, define
\begin{equation*}
	g_{k+1}(z) = 
\left\{
	\begin{array}{ll}
		P_k(z),	&	z\in E_k,\\
		f_{k+1}(z),	&	z\in \widetilde{B}_{k+1},	
	\end{array}
\right.
\end{equation*}
which is a holomorphic function on $E_k\cup \widetilde{B}_{k+1}$. Since $E_k\cup \widetilde{B}_{k+1}$ is a polynomially convex set there exists a holomorphic polynomial $P_{k+1}$ such that 
$$|P_{k+1}(z)-g_{k+1}(z)|<\frac{\epsilon_{k+1}}{2^{k+1}}, \quad \textit{ for all }\quad z\in E_k\cup \widetilde{B}_{k+1}.$$
Equivalently, 
\begin{equation}\label{induction}
\begin{split}
|P_{k+1}(z)-P_{k}(z)|&<\frac{\epsilon_{k+1}}{2^{k+1}}, \quad \textit{ for all }\quad z\in E_{k},\\
|P_{k+1}(z)-f_{k+1}(z)|&<\frac{\epsilon_{k+1}}{2^{k+1}}, \quad \textit{ for all }\quad z\in \widetilde{B}_{k+1}.\\
\end{split}
\end{equation}

We claim that the sequence $\{P_k\}$, obtained by the above procedure, is uniformly Cauchy on compact subsets of $\C^n$. To prove this, let $K$ be a compact subset of $\C^n$ and fix $\epsilon>0$. Recall that $(B_k)_{k=1}^\infty$ is a serpentine tilling, so for sufficiently large $N$, we have $K\subset \cup_{k=1}^NB_k=E_N.$ Choose such an $N$ so large that $\epsilon_N<\epsilon.$ Let $N\leq s\leq \ell$. Then for $z\in K$ one has $z\in E_s\subset E_\ell$.  So on $K,$
$$
	|P_\ell-P_s| = \left|\sum_{j=s+1}^\ell( P_j-P_{j-1})\right| < \sum_{j=s+1}^\ell\frac{\epsilon_{j}}{2^{j}}
	<\sum_{j=s+1}^\infty\epsilon_j<\epsilon_s < \epsilon_N<\epsilon.
$$
Thus, the sequence $(P_k)_{k=1}^\infty$ is indeed uniformly Cauchy on compacta and hence converges to an entire function $g$.

There only remains to show that $g$ performs the required approximation of $(f_k)_{k=1}^\infty$ on $(\widetilde B_k)_{k=1}^\infty$. 
Fix $k.$ First,  note that $\widetilde{B}_k\subset E_{j-1}\subset E_j$ for every  $j >k$. By  (\ref{induction}) and the above computation, for $z\in\widetilde{B}_k$, one has
\begin{equation*}
\begin{split}
|g(z)-f_k(z)|&\leq |g(z)-P_k(z)|+|P_k(z)-f_k(z)|\\
&=\lim_{j\to\infty}|P_j(z)-P_k(z)|+|P_k(z)-f_k(z)|\\ 
&\leq\sum_{j=k+1}^\infty |P_j(z)-P_{j-1}(z)|+\frac{\epsilon_k}{2^k}\\
&\leq\sum_{j=k+1}^\infty \frac{\epsilon_j}{2^j}+\frac{\epsilon_k}{2^k},\,\,( \textit{by (\ref{induction}) and since } z\in E_{j} )\\
&\leq \left(\frac{1}{2^{k}}\sum_{j=k+1}^\infty \epsilon_j\right)+\frac{\epsilon_k}{2^k}\leq\frac{\epsilon_{k}}{2^k}+\frac{\epsilon_k}{2^k}\le\epsilon_k.\\
\end{split}
\end{equation*}
Therefore, 
$$|g(z)-f_k(z)|\leq \epsilon_k, \textit{ on } \widetilde{B}_k.$$
\end{proof}

\begin{lemma}\label{g-u} 
Let $v$ be a continuous function on $\C^n$, and $\delta$ a positive (decreasing) continuous function on $[0,+\infty)$. 
Then, there is a Masonic temple $E\subset \C^n,$ and an entire function $g,$ such that 
\begin{equation}\label{eta}
	\mu_*\{z\in\C^n\setminus E: |z|>r\}<\delta(r), \quad \mbox{for all} \quad r\in [0,+\infty); 
\end{equation}
\begin{equation}\label{delta}
|g(z)-v(z)|<\delta(|z|), \quad \mbox{for all} \quad  z\in E;
\end{equation}
and, for every $p\in \partial U$ and $\lambda>0$, there exists an $r_{p,\lambda}>0$ such that
\begin{equation}\label{mu}
	\mu\big(\B(p, r)\cap(U\setminus h^{-1}(E))\big) \le 
		\lambda\cdot\mu\big(\B(p, r)\cap U\big), \quad \mbox{for all} \quad  0<r<r_{p, \lambda}.
\end{equation}
\end{lemma}

\begin{proof}
Suppose a serpentine tiling $(B_k)_{k=1}^\infty$ of $\C^n$ is given. 
For $k=1,2,\ldots,$ set 
$$
	\delta_k=\min\{\delta(r_k), \delta(k+1)\}, \quad \mbox{where} \quad
	r_k=\max\left\{|z| : z\in B_1\cup\cdots\cup B_k\right\}.
$$
For each $k,$ $r_k\le r_{k+1}$ so $\delta_k\ge\delta_{k+1}$.

For fixed  $k=1, 2, \dots$, the restriction of $v$ to $B_k$ is uniformly continuous  (by compactness of $B_k$). Therefore, for each $k$ there exists $a_k>0$ such that 
\begin{equation}
|v(z)-v(w)|<\delta_k/2, \vspace{.5cm} \textit{ if } z, w\in B_k, \textit{ and }  d(z, w)<a_k.
\end{equation} 

Consider a new sequence $(\theta_k)$, where 
$$\theta_k:=\min \{ a_k, {\delta}_k\}.$$
By Lemma \ref{9-12} there exists a partition $(\beta_{i, j}^k)$ of $B_k$  such that 
$diam(\beta^k_{i,j})<\theta_k$, for each $k=1, 2, \dots$. Furthermore, there exists a stratification
$\widetilde B_k=\cup\,\widetilde\beta^k_{i,j}$ of $B_{k}$, such that, for the Masonic tiling $E=\cup_k \widetilde B_{k}$,  
$$
	\mu_*\left\{z\in\C^n\setminus E: |z|>k\right\}<\theta_k, \quad k=1, 2, \dots;
$$ 
for all $p\in \partial U$ and for all $\lambda>0$, there exists an $r_{p, \lambda}>0$, such that 
$$
	\mu\big(\B(p, r)\cap(U\setminus h^{-1}(E))\big) \le 
		\lambda\cdot\mu\big(\B(p, r)\cap U\big), \quad \mbox{for all} \quad  0<r<r_{p, \lambda}.
$$
which proves (\ref{mu}). Also, denoting by $[r]$ the greatest integer less than or equal to $r$, 
$$
\mu_*\left\{z\in\C^n\setminus E: |z|>r\right\}\le\mu_*\left\{z\in\C^n\setminus E: |z|> [r]\right\}<\theta_{[r]}\leq \delta_k\leq\delta([r]+1)\le \delta(r),
$$
which proves (\ref{eta}).

Define a holomorphic function on this Masonic temple $E=\cup_k \widetilde{B}_{k},$  associated to the continuous function $u$ on $\C^n$. Pick sample points $z^k_{i, j}\in\widetilde{\beta}^k_{i,j}\subset \widetilde{B}_k$ for each element $\widetilde B_k$ in the Masonic tiling. Define $f_k$ on $\widetilde B_{k}$ by setting 
$$f_k(z)=v(z^k_{i, j}), \quad \textit{ if }  z\in \widetilde{\beta}^k_{i,j};$$
that is, $f_k$ has the constant value $v(z^k_{i, j})$ on $\widetilde{\beta}^k_{i,j},$ for each $i, j$. It is clear that $(f_k)_{k=1}^\infty$ is a sequence of holomorphic functions such that $f_k$ is holomorphic on $ \widetilde{B}_{k}$. For this sequence of holomorphic functions and the sequence $(\delta_k)_{k=1}^\infty$, by Theorem \ref{masonic approximation}, there exists an entire function $g$ such that 
$$|g(z)-f_k(z)|<\delta_k/2\quad \textit{ for } \quad z\in \widetilde{B}_k.$$
We note that if $z\in \widetilde{\beta}^k_{i,j}$, then $|v(z^k_{i, j})-v(z)|<\delta_k/2$ because $diam(\widetilde{\beta}^k_{i,j})\leq diam (\beta^k_{i,j})<\theta_k\leq a_k$.

Suppose $z\in E$. Then there exists $k=1, 2, \dots$ and $i=i_k, j=j_k$ such that $z\in \widetilde{\beta}_{i, j}^k\subset \widetilde{B}_k\subset B_k$. One, therefore, has
\begin{equation*}
\begin{split}
|g(z)-v(z)|&\leq |g(z)-f_k(z)|+|f_k(z)-v(z)|\leq \delta_k/2+|v(z^k_{i, j})-v(z)|\\
&\leq \delta_k/2+\delta_k/2=\delta_k\leq\delta(r_k)\leq\delta(|z|),
\end{split}
\end{equation*}
where the last inequality follows from the property that $\delta$ is a decreasing function.
\end{proof}

The following approximation result is not needed in the proof of our main theorem, but is closely related to the topic and of independent interest.

\begin{theorem}\label{measurable}
In $\C^n,$  for every regular Borel measure $\eta,$ for every   Borel measurable function $\varphi,$
and for every positive decreasing continuous function $\epsilon$  on $[0,+\infty),$ there exists an entire  function $f$  and a closed set $E\subset \C^n,$ such that 
$$
	\eta\{z\in \C^n\setminus E: |z|>r\} \,\, < \,\, \epsilon(r)
$$
and
$$
	|f(z)-\varphi(z)|<\epsilon(|z|),  \quad \mbox{for all} \quad z\in E.
$$
\end{theorem}

\begin{proof} We may assume that $\epsilon$ deceases to zero.
For each $k=1,2,\ldots,$ choose a positive number $\epsilon_k$ with $\epsilon_k<\epsilon(k+2)/2.$ We may further assume that,  
$$\epsilon_{k+1} +\epsilon_{k+2}+\cdots<\epsilon_k,
$$
for each $k.$ Define 
$$
A_k=\left\{z\in \C^n : k-1\le |z|\le k\right\} \quad ; \quad k=1, 2, \ldots. 
$$
Since $A_k$ is locally compact, it follows from  Lusin's theorem \cite[Theorem 2.24]{R}, that there is a continuous function $u_k$ on $A_k,$ and a subset (which by the regularity of $\eta$ we may take to be compact) $K_k\subset A_k,$ such that $u_k=\varphi$ on $K_k$ and $\mu(A_k\setminus K_k)<\epsilon_k.$ 
Define $u$ on  $E_1=\cup_kK_k$ by setting $u=u_k$ on each $K_k.$
Since the family of $K_k$'s is a locally finite family of disjoint compacta, the set $E_1=\cup_kK_k$ is closed and $u$ is continuous on $E_1.$  By Tietze's theorem \cite[Theorem 20.4]{R}, one may extend $u$ to a continuous function on $\C^n$, which we continue to denote by $u.$ Furthermore,
\begin{equation*}
\begin{split}
\eta\{z\in \C^n\setminus E_1: |z|>r\}		&\le \eta\left(\bigcup_{k=[r]}^\infty(A_k\setminus K_k)\right)\\									&\le\sum_{k=[r]}^\infty\epsilon_k\le\epsilon_{[r]-1} <\epsilon([r]+1)/2\le\epsilon (r)/2.
\end{split}
\end{equation*}

On the other hand, by Theorem \ref{g-u}, there exists an entire  function $f$  and a closed set (Masonic temple) $E_2\subset \C^n,$ such that 
$$
	\eta\{z\in \C^n\setminus E_2: |z|>r\} < \epsilon(r)/2
$$
and
$$
	|f(z)-u(z)|<{\epsilon(|z|)},  \quad \mbox{for all} \quad z\in E_{2}.
$$

Now for $z\in E_1\cap E_2$ one has 
$$
	|f(z)-\varphi(z)|=|f(z)-u(z)|<\epsilon(|z|).
$$
Furthermore, by subadditivity of $\eta$, for $E=E_1\cap E_2$ one has
\begin{equation*}
\begin{split}
	\eta\{z\in \C^n\setminus E: |z|>r\} & \leq  \eta\{z\in \C^n\setminus E_1: |z|>r\} +\eta\{z\in \C^n\setminus E_2: |z|>r\}\\
	&< \epsilon(r)/2+\epsilon(r)/2=\epsilon(r).
\end{split}
\end{equation*}

Therefore, $E$ is the desired closed set and $f$ is the desired entire function. 
\end{proof}	

{\bf Remark.}
For Lebesgue measure, the proof of Theorem \ref{g-u} can be somewhat simplified, taking into account Lemma 2.4 in \cite{I} which states that for an arbitary polynomially convex compact set $Y\subset \C^n$ and $\epsilon>0,$ there is a totally disconnected polynomially convex compact set $K\subset Y,$ such that $m(Y\setminus K)<\epsilon$ and also taking into account, that polynomials are dense in $C(K),$ when $K$ is a totally disconnected polynomially convex compact set \cite[Chapter 8, page 48]{AW}.
\section{Proof of Theorem \ref{main}}

Recalling that our manifold $M$ is equipped with a distance $d,$ we denote by $\B(p,r)$ the open ball of center $p\in M$ and radius $r>0.$

\begin{lemma}\label{bundleQK}
Let  $K\subset \partial_0 U$ and  $
Q\subset\partial U$ be disjoint compact sets.
Then, for each $\varepsilon>0,$ there exists $\delta>0$ and an open neighbourhood $V_\delta$ of $Q$ in $M$ contained in a $\delta$-neighbourhood 
 of $Q$ in $M$ disjoint from $K$ such that 
\[
	\mu\big(\B(p, r)\cap U\cap V_\delta\big) \le 
		\varepsilon\cdot\mu\big(\B(p, r)\cap U\big), \quad 
	\forall p\in K, \quad \forall r>0.
\]
Furthermore,  $\delta$ can be chosen so that  $\B(p,r)\cap V_\delta=\emptyset$ for all $p\in K$ and $r<\delta$.
\end{lemma}

\begin{proof} 
Let $r_0=d(Q,  K)/2$ which is not zero, since  $Q$  and $K$ are compact. Define 
$$
\rho =\min_{q\in K}\mu\left(\B(q, r_0)\cap U\right).
$$
We claim $\rho>0$, otherwise there exists a sequence of $(p_n)\subset K$ such that it converges to a point $p\in K\subset \partial_0U$ (by compactness of $K$) and 
$$
\mu\left(\B(p_n, r_0)\cap U\right)<\frac{1}{n}.
$$
For a given $\eta>0$ one can choose $n_\eta$ large enough such that $1/n<\eta$ and $d(p_n, p)<r_0/2$ for every $n>n_\eta$. This implies that 
$$
\mu\left(\B(p, r_0/2)\cap U\right)\leq \mu\left(\B(p_n, r_0)\cap U\right)\le \eta;
$$
that is, $\mu\left(\B(p, r_0/2)\cap U\right)=0$. This contradicts our assumption that $p\in \partial_0U$, so $\rho\neq 0$. 

We now let $0<\delta<r_0$. Since the measure $\mu$ is regular, we may choose $V_\delta,$ such that $\mu\left(V_\delta\cap U\right)<\epsilon\cdot\rho$. Let $r>0$ be arbitrary. If $r<r_0$, then the claim is true. Indeed,
\[
	\mu\big(\B(p, r)\cap U\cap V_\delta\big) \le \mu\big(\B(p, r)\cap V_\delta\big)=0=
		\epsilon\cdot\mu\big(\B(p, r)\cap V_\delta\big), \quad 
	\forall p\in  K,
\]
since $\B(p, r)\cap V_\delta=\emptyset$ for every $p\in  K$; if $r>r_0$, one has
$$
\mu\big(\B(p, r)\cap U\cap V_\delta\big) \le \mu\big(U\cap V_\delta\big)\le \epsilon \cdot \rho\le \epsilon \cdot \mu\left(\B(p, r_0)\cap U\right) \le \epsilon \cdot \mu\left(\B(p, r)\cap U\right) 
$$
for every $p\in  K$. This completes the proof of the first part.

The second part of the lemma is clear. 
\end{proof}

\begin{lemma}\label{continuous}
There exists a measurable   set $F,$ with $S\subset F\subset  \partial U$ and $\nu(\partial U\setminus F)=0;$ 
there exists a function $u\in C(U)$, continuous on $U,$ such that,
for every $p\in F$ there exists $\mathcal E_p\subset U$ such that 
\begin{equation}\label{u to psi}
	u(x)\to \psi(p), \quad  \mbox{as} \quad  x\to p, \quad x\in \mathcal E_p
\end{equation} 
and for every $\lambda>0,$ there exists $r_{p, \lambda}>0,$ such that
\begin{equation}\label{mu mu epsilon}
	\mu\big( (\B(p, r)\cap U)\setminus \mathcal E_p\big)\le \lambda\cdot\mu\big( \B(p, r)\cap U\big), 
		\quad \mbox{for all} \quad r<r_{p, \lambda}.
\end{equation}
\end{lemma}

\begin{proof} 
 Put $S_0=S\cap \partial_0 U$ and $S^0=S\setminus \partial_0U,$ so $S=S^0\dot\cup S_0.$ We can write $S_0=\cup_{n=1}^\infty S_n$ and $S^0=\cup_{n=1}^\infty S^n,$ where $(S_n)_{n=1}^\infty$ and $(S^n)_{n=1}^\infty$ are increasing sequences of compact sets. 

First we assume that $\nu(\partial_0 U\setminus S_0)<+\infty$. 
By Lusin's theorem (see \cite{R} and \cite[Theorem 2]{W}), there exists a compact set $Q_1$  in $\partial_0 U\setminus S_0$ such that $\nu(\partial_0 U\setminus S_0\setminus Q_1)<2^{-1}$ and the restriction of $\psi$ to $Q_1$ is continuous. Now, again by Lusin's theorem, we can find a compact set $Q_2$ in $ \partial_0 U\setminus S_0\setminus Q_1$ with $\nu\big((\partial_0 U\setminus S_0\setminus Q_1)\setminus Q_2\big)<2^{-2}$ so that the restriction of $\psi$ to 
$Q_2$ is continuous. Since, $Q_1$ and $Q_2$ are disjoint compact sets, the restriction of $\psi$ to $Q_1\,\dot\cup \,Q_2$ is continuous. By induction we can construct a sequence of compact sets $(Q_n)_{n=1}^\infty$ so that $Q_{n}\subset {\partial_0 U}\setminus S_0\setminus \cup_{j=1}^{n-1} Q_j$, $\nu(\partial_0 U\setminus S_0\setminus \cup_{j=1}^{n} Q_j)<2^{-n}$ and 
the restriction of $\psi$ to $Q_1\,\dot\cup\cdots\dot\cup\, Q_n$ is continuous, for $n=1,2,3,\ldots$. 
It is obvious that $(Q_n)_{n=1}^\infty$ is a family of pairwise disjoint compact sets and setting
\begin{equation}\label{Q}
	Q=Q_1\cup Q_2\cup\cdots,
\end{equation}
we have that $Q\subset \partial_0 U$ and $\nu\big(({\partial_0 U}\setminus S_0)\setminus Q\big)=0.$

If $\nu({\partial_0 U}\setminus S_0)=+\infty,$ then by the $\sigma$-finiteness of the measure $\nu$, there exists a pairwise disjoint sequence of measurable sets $(R_l)_{l=1}^\infty$ with $\nu(R_l)<+\infty$ and 
$\partial_0 U\setminus S_0=\dot{\bigcup}_{l=1}^\infty R_l$. By the previous argument applied to the function $\psi$ restricted to the set $R_l$ we can find a pairwise disjoint sequence of compact sets $(Q_{n, l})_{n=1}^\infty$ of $R_l,$ so that the restriction of $\psi$  to each   $Q_{n, l}$ is continuous and $\nu(R_l\setminus \dot\cup_{n=1}^\infty Q_{n, l})=0.$ We may arrange the countable family of pairwise disjoint compact sets 
$$\{Q_{n,\ell}: n=1,2,\ldots; \ell=1,2,\ldots,\}$$
in a sequence, so that, whether or not $\nu({\partial_0 U}\setminus S_0)=+\infty,$ we may write $Q$ in the form (\ref{Q}), where the $Q_n$'s are disjoint and the restriction of $\psi$ to each $Q_n$ is continuous and $\nu\big(({\partial_0 U}\setminus S_0)\setminus Q\big)=0.$

Similarly, there is a set $A\subset (\partial U\setminus\partial_0U)\setminus S^0$ of the form $A=A_1\cup A_2\cdots,$ where $(A_n)_{n=1}^\infty$ is a sequence of disjoint compact sets, $\psi$ restricted to each $A_n$ is continuous  and
$$
	\nu\big(((\partial U\setminus\partial_0U)\setminus S ^0)\setminus A\big)=0.
$$

Setting $F=A\cup S\cup Q,$ we have $S\subset F\subset \partial U$ and 
$\nu\big(\partial U\setminus F)=0.$   Note that the restriction of $\psi$ to each $F_n=A_1\cup\cdots\cup A_n\cup S\cup Q_1\cup\cdots\cup Q_n$ is continuous, because the restriction to the closed set $S$ is continuous, the restriction to the compact set $A_1\cup\cdots\cup A_n\cup Q_1\cup\cdots\cup Q_n$ is continuous, and this compact set is disjoint from the closed set $S,$ so the restriction to the union, which is $F_n,$ of the closed set and the compact set,  is also continuous.

We now begin to extend the function $\psi$. For this we shall construct inductively an increasing sequence $(E_n)_{n=1}^\infty$ of subsets of $U$ and a sequence of functions $(f_n)_{n=1}^\infty$.  

By Lemma \ref{bundleQK}, for $l=2,3,\ldots,$ there is an open neighbourhood $V_{1,l}$ of $Q_l$ contained in a $\delta_{1,l}$-neighbourhood of $Q_l$ in $M$ such that
\begin{equation}
\label{neigVj}
\begin{aligned}
	\mu\big(\B(p, r)\cap U\cap V_{1,l}\big) &\le \frac{1}{2^l}\cdot \mu\big(\B(p, r)\cap U\big), \quad 
	\forall p\in S_1\cup Q_1,\\
\mu\big(\B(p, r)\cap U\cap V_{1,l}\big)&=0, \quad \forall p\in S_1\cup Q_1, \quad \forall r<\delta_{1,l}.
\end{aligned}
\end{equation}
Without loss of generality we can assume that $\delta_{1,l}<1$ for $l=2,3,\ldots .$ Set
\[
E_1=U\setminus \cup_{l=2}^\infty V_{1,l}.
\]
Since $F_1$ is closed in $E_1\cup F_1,$ by the Tietze extension theorem we can extend the function $\psi_1$ to a continuous function $f_1$ on $E_1\cup F_1.$

Set $E_0=\emptyset$ and assume that for $j=1,\ldots,n,$ we have fixed  $E_j,$ $\delta_{j,l},$ for $l\ge j+1,$ and functions $f_j$ continuous on $E_j\cup F_j,$
$$
	f_j(x) = 
\left\{
	\begin{array}{ll}
		f_{j-1}(x),	&	x\in E_{j-1},\\
		\psi(x),	&	x\in F_j,	
	\end{array}
\right.
$$
where $E_j=U\setminus \cup_{l=j+1}^\infty V_{j,l}$ with $V_{j,l}$ being an open neighbourhood of $Q_l$ contained in a $\delta_{j,l}$-neighbourhood  of $Q_l$ in $M\setminus E_j$ such that
\begin{equation}
\label{neighbours}
\begin{aligned}
	\mu\big(\B(p, r)\cap U\cap V_{j, l}\big)&\le \frac{1}{2^l}\cdot \mu\big(\B(p, r)\cap U\big), \quad 
	\forall p\in \cup_{k=1}^j(S_k\cup Q_k), \\
	\mu\big(\B(p, r)\cap U\cap V_{j, l}\big)&=0, \quad \forall p\in \cup_{k=1}^j(S_k\cup Q_k), \quad \forall r<\delta_{j,l},
\end{aligned}
\end{equation}
$\delta_{j,l}<1/j,$   for $j=1,\ldots,n$ and $l=j+1,j+2,\ldots$. For the step $n+1$, using Lemma \ref{bundleQK} again we have that for every natural number $l>n+1$  there is an open neighbourhood $V_{n+1,l}$ of $Q_l$ contained in a $\delta_{n+1,l}$-neighbourhood of $Q_l$ in $M\setminus E_n$ such that
\begin{equation*}
\begin{aligned}
	\mu\big(\B(p, r)\cap U\cap V_{n+1,l}\big) &< \frac{1}{2^l}\cdot \mu\big(\B(p, r)\cap U\big), \quad 
	\forall p\in \cup_{k=1}^{n+1}(S_k\cup Q_k), \\
	 \mu\big(\B(p, r)\cap U\cap V_{n+1,l}\big)&=0, \quad \forall p\in \cup_{k=1}^{n+1}(S_k\cup Q_k), \quad \forall r<\delta_{n+1,l}.
\end{aligned}
 \end{equation*}
Without loss of generality we can assume that $\delta_{n+1,l}<1/(n+1).$ Set
\[
E_{n+1}=U\setminus \cup_{l=n+2}^\infty V_{n+1,l}.
\]

Note that $E_n\cup F_{n+1}$ is relatively closed in $E_{n+1}\cup F_{n+1}$. Furthermore, the function $f_{n+1}$ defined on  $E_n\cup F_{n+1}$ as $f_{n+1}=f_{n}$ on $E_n\cup F_n$ and $f_{n+1}=\psi_n$ on $Q_{n+1}$ is continuous, since these two closed sets $E_n\cup F_n$ and  $Q_{n+1}$ are disjoint. Therefore, by the Tietze extension theorem we can extend the function $f_{n+1}$ to a continuous function on $E_{n+1}\cup F_{n+1}$ that we denote in the same way. 

Note that $\cup_{n=1}^\infty E_n=U$, since  $\delta_{n+1,l}<1/(n+1).$ Then, the function $u,$ defined on $U$ as $u(x)=f_n(x)$ if $x\in E_n,$ is continuous on $U.$

Finally, we show  the limit property.  
Fix $p\in F.$  If $p\in \partial U\setminus \partial_0 U,$ we set $\mathcal E_p=\emptyset,$ so (\ref{u to psi}) and (\ref{mu mu epsilon}) are trivial because
$$
	\mu\big(\B(p,r)\cap (U\setminus \mathcal E_p)\big)=\mu\big(\B(p,r)\cap U\big)=0, \quad \mbox{ for all }  r>0.
$$ 
If $p\in F\cap \partial_0U,$ then,
either $p\in S$ and there is some $n(S),$ such that $p\in S_n,$ for all $n\ge n(S),$ or $p\in Q,$ in which case  $p\in Q_n,$ for some $n=n(Q).$ In the first case, set $n(p)=n(S);$ in the second case, set $n(p)=n(Q)$ and in either case set $\mathcal E_p=E_{n(p)}.$

We have that $p\in F_{n(p)}, \, u=f_{n(p)}$ on $E_{n(p)}=\mathcal E_p,$ $f_n$ is continuous on $E_n\cup F_n$ and $f_{n(p)}=\psi_{n(p)}$ on $F_{n(p)}.$  Therefore, $u(x)\to \psi(p),$ as $x\to p$ in $\mathcal E_p=E_{n(p)}.$ 

For fixed $\lambda>0,$ choose $m>n(p)$ so large that $1/2^m<\lambda$ and set 
$$r_{p,\lambda}=\min\left\{\delta_{n(p),\ell} : n(p)<\ell\le m\right\}.$$
Then, for $r<r_{p,\lambda},$ one has
\begin{equation*}
\begin{aligned}
	\mu\big((\B(p,r)\cap U)\setminus\mathcal E_p\big) &=
	\mu\big(\B(p,r)\cap U\cap \cup_{\ell=n(p)+1}^\infty V_{n(p), \ell}\big)\le\\
	\sum_{\ell=n(p)+1}^\infty \mu\big(\B(p,r)\cap U\cap  V_{n(p),\ell}\big) &=
	\sum_{\ell=m+1}^\infty \mu\big(\B(p,r)\cap U\cap  V_{n(p),\ell}\big)\le\\
	\sum_{\ell=m+1}^\infty\frac{1}{2^\ell}\mu\big(\B(p,r)\cap U\big)&=
	\frac{1}{2^m}\mu\big(\B(p,r)\cap U\big)<\lambda\cdot\mu\big(\B(p,r)\cap U\big).
\end{aligned}
 \end{equation*}
\end{proof}
Now, we are in a position to give a proof of Theorem \ref{main}.
\noindent

\bigskip
\noindent
{\bf Proof of Theorem 2.} 

\smallskip
By Lemma \ref{continuous} there exists a measurable set $F$, with 
$$S\subset F\subset \partial U \quad \textit{ and }\quad \nu(\partial U\setminus F)=0,$$
a  continuous function $u$ on $U$  such that for every $p\in F$, there exist a set $\mathcal E_p\subset U$ such that 
\begin{equation}\label{limit1}
u(x)\rightarrow \psi(p), \quad \mbox{ as } x \rightarrow p \quad \mbox{with} \quad x\in \mathcal E_p.
\end{equation}
Furthermore, for $\lambda>0,$  there is an $r_{p,\lambda}>0,$ such that
\begin{equation}\label{E_1}
	\mu\big( (\B(p, r)\cap U)\setminus \mathcal  E_p\big)\le \lambda\cdot\mu\big( \B(p, r)\cap U\big), 
		\quad \mbox{for all} \quad r<r_{p,\lambda}.
\end{equation}

Now let $h:U\to \C^n$ be a proper embedding of the Stein domain $U$ into some complex Euclidean space $\C^n,$ and consider the continuous function $v=u\circ h^{-1},$ defined on the closed subset $h(U)\subset\C^n$. We extend $v$ continuously to all of $\C^n.$
Note that, for $q\in \partial U$, if $x\rightarrow q$, $x\in U$, then $h(x)\rightarrow \infty$.

Let $\delta$ be a positive (decreasing) continuous function on $[0, +\infty),$ such that $\delta(r)\searrow 0$ as $r\nearrow +\infty$. Let 
$(\widetilde B_k)_{k=1}^\infty$ be a stratification of $( B_k)_{k=1}^\infty,$ such that the Masonic temple $E_{2}=\cup_k  \widetilde B_k$ satisfies Lemma \ref{g-u}.
Therefore, by Lemma \ref{g-u}, the closed set $E_{2}\subset \C^n$ may be constructed such that, setting $\mathcal E_2=h^{-1}(E_2),$
for every $p\in\partial U$ and, for every $\lambda>0,$ there is an $r^\prime_{p,  \lambda}>0$, such that
\begin{equation}\label{E_2}
	\mu\big( (\B(p,r)\cap U)\setminus \mathcal E_2\big)\le \lambda\cdot\mu\big( \B(p,r)\cap U\big), 
		\quad \mbox{for all} \quad r<r^\prime_{p, \lambda}.
\end{equation}
and there exists an entire function $g$, such that
$$|g(z)-v(z)|<\delta (|z|), \quad \textit{ for all } z\in E_{2}.$$
Note that for $z=h(x)$, where $x$ is in $\mathcal E_{2}\subset U$, the above can be written as 
$$|g(h(x))-u(x)|<\delta (|h(x)|).$$

Define $f=g\circ h$ which is a holomorphic function on $U$. Furthermore, for $p\in F$ and $x\in \mathcal E_{2}$ arbitrary, one has
\begin{equation*}
\begin{split}
|f(x)-\psi(p)|&=|(g \circ h)(x)-\psi(p)|\\
&\leq |g ( h(x))-u(x)|+|u(x)-\psi(p)|\\
&\leq \delta( |h(x)|)+|u(x)-\psi(p)|.\\
\end{split}
\end{equation*}
The above (and the properties of the  functions $h$ and $\delta$) implies that, for $p\in F,$ 
\begin{equation}\label{limit2}
	f(x)\to \psi(p) \quad \mbox{as} \quad x\to p,  \quad x\in \mathcal E_2.
\end{equation}

We now let $\mathcal F_p=\mathcal E_p\cap \mathcal E_2$, for $p\in F$. Then $\mathcal F_p$ is closed, furthermore, 
$$f(x)\rightarrow \psi(p) \quad \mbox{as} \quad x\rightarrow p,  \quad x\in\mathcal F_p.$$ 
by Equation (\ref{limit2}). Finally, for a given $\lambda>0$, by invoking (\ref{E_1}) and (\ref{E_2}), one has
\begin{equation*}
\begin{split}
\mu\big( (\B(p,r)\cap U)\setminus \mathcal F_p\big)&\le
	\mu\big( (\B(p,r)\cap U)\setminus \mathcal E_2\big)+
	\mu\big( (\B(p,r)\cap U)\setminus \mathcal E_p\big)\\
	&\le \lambda \cdot \mu\big( \B(p,r)\cap U\big)+\lambda \cdot \mu\big( \B(p,r)\cap U\big)\\	
	&\le 2\,\lambda \cdot \mu\big( \B(p,r)\cap U\big)\\
\end{split}
\end{equation*}		
for all $0<r<\min \{r_{p, \lambda}, r'_{p, \lambda} \}$.

This concludes the proof of Theorem 2.
\medskip

{\it Acknowledgement.}
The junior author would like to thank the senior author for introducing the problem and for his generous support. We thank the authors of \cite{FG3}, our Lemma \ref{bundleQK} is inspired by a similar result from their manuscript. We also thank Alexander Izzo for drawing our attention to his lovely paper \cite{I}, which, as we remarked, makes it possible to simplify our proof of Theorem \ref{g-u} in the case of Lebesgue measure. We thank Anush Tserunyan for stimulating discussions regarding descriptive set theory.


\end{document}